\documentclass{article}

\usepackage{authblk}

\usepackage[utf8]{inputenc}%       encodage du fichier source
\usepackage[T1]{fontenc}%          gestion des accents (pour les pdf)
\usepackage[english]{babel}
\usepackage{textcomp}%             caractères additionnels
\usepackage{amsmath,amssymb,amsthm}% packages ams
\usepackage{stmaryrd} %pour les intervalles d’entiers
\usepackage{lmodern}%              remplacer éventuellement par txfonts, fourier, etc.
\usepackage[a4paper]{geometry}%    taille correcte du papier
\usepackage{xcolor, pict2e}%               pour gérer les couleurs
\usepackage{microtype}%            améliorations typographiques
\usepackage{caption}
\usepackage{listings}
\usepackage{multicol}
\usepackage{moreverb}
\usepackage{hyperref}%             gestion des hyperliens
\hypersetup{pdfstartview=XYZ}%     zoom par défaut
\usepackage{wrapfig}%				texte coulant autour des figures
\usepackage[sans]{dsfont}

\usepackage{mathrsfs}
\usepackage{tikz, pgfplots}
\usetikzlibrary{positioning}
\usepackage{amsmath,amssymb}
\usetikzlibrary{decorations.pathreplacing}

\usepackage{ytableau}

\usepackage{mathdots}
\usepackage{hyperref}%             gestion des hyperliens
\hypersetup{
    colorlinks=true,
    linkcolor=blue,
    citecolor=magenta,
    urlcolor=blue,
    pdfborder={0 0 0}
}

\definecolor{fond}{rgb}{0.05,0.05,0.25}

\DeclareMathOperator{\Cay}{Cay}

\DeclareMathOperator{\ch}{\mathrm{ch}}
\DeclareMathOperator{\Id}{\mathrm{Id}}

\newcommand{\Hap}{\ensuremath{\Xi}}
\newcommand{\hap}{\ensuremath{\xi}}

\newcommand{\kS}{{\ensuremath{\mathfrak{S}}} }

\newcommand{\kY}{{\ensuremath{\mathfrak{Y}}} }

\usepackage[font=sf, labelfont={sf,bf}, margin=1cm]{caption}
\usepackage{enumitem}
\setlist[itemize,1]{nosep}
\setlist[enumerate,1]{itemsep=0pt,label=(\alph*)}

\usepackage{float}
\usepackage[caption = false]{subfig}
\usepackage{graphicx}

\newtheorem*{theorem*}{Theorem}
\newtheorem{theorem}{Theorem}[section]
\newtheorem{proposition}[theorem]{Proposition}
\newtheorem{lemma}[theorem]{Lemma}
\newtheorem{corollary}[theorem]{Corollary}

\newtheorem{remark}[theorem]{Remark}

\usepackage{todonotes}

\newcommand{\cS}{{\ensuremath{\mathcal S}} }
\newcommand{\cT}{{\ensuremath{\mathcal T}} }

\newcommand{\bbR}{{\ensuremath{\mathbb R}} }

\newcommand{\bbP}{{\ensuremath{\mathbb P}} }

\newcommand{\ag}{\left\{ } %accolade gauche
\newcommand{\ad}{\right\} }

\newcommand{\pg}{\left(} %parenthèse gauche
\newcommand{\pd}{\right)}
\newcommand{\bg}{\left|}
\newcommand{\bd}{\right|}
\newcommand{\lf}{\left\lfloor}
\newcommand{\rf}{\right\rfloor}
\newcommand{\lc}{\left\lceil}
\newcommand{\rc}{\right\rceil}

\newcommand{\du}{{\ensuremath{\;:\;}}} %\dp n'est pas accepté

\DeclareMathOperator{\sgn}{sgn}

\makeatletter
\newcommand*\bigcdot{\mathpalette\bigcdot@{.5}}
\newcommand*\bigcdot@[2]{\mathbin{\vcenter{\hbox{\scalebox{#2}{$\m@th#1\bullet$}}}}}
\makeatother

\numberwithin{equation}{section}

\pgfplotsset{compat=1.18}

\usepackage{todonotes}
\usepackage{comment}

\newcommand{\bigpar}[1]{\left({#1}\right)}
\def \l{\lambda}

\def\qnl{q_{n,\ell}}
\def\Sn{\ensuremath{\mathfrak{S}_n}}

\def\nto{n\to\infty}

\def\lln{\lambda\vdash n}
\begin{document}

\renewcommand{\theparagraph}{\thesubsection.\arabic{paragraph}} 
\title{Very high genus double Hurwitz numbers via mixing times}

\author[1]{Baptiste Louf}
\author[2]{Lucas Teyssier}
\affil[1]{CNRS and Université de Toulouse,
\texttt{baptiste.louf@math.univ-toulouse.fr}}
\affil[2]{Université de Lorraine, \texttt{lucas.teyssier@univ-lorraine.fr}}
%Université de Lorraine, CNRS, IECL, F-54000 Nancy
\date{}

\maketitle
\begin{abstract}
    We obtain very high genus asymptotics for double Hurwitz numbers when both the size $n$ and the genus $g$ go to infinity, under the constraint that $g$ grows much faster than $n$, more precisely that $\frac{2g}{n}-\ln n\to\infty$. Our proof relies on adapting arguments about the mixing time of the transposition random walk introduced by Diaconis and Shahshahani.
\end{abstract}

\section{Introduction}
Hurwitz numbers count (equivalence classes of certain) ramified covers of the sphere by a surface of higher genus, or equivalently, factorizations of the identity in the symmetric group with constraints on the factors. They are connected to several research areas, such as the intersection theory of the moduli space of curves (via the ELSV formula~\cite{ELSV}), integrable hierarchies~\cite{Ok} or topological recursion~\cite{EMS}. They belong to the realm of \emph{enumerative geometry}, which consists, roughly speaking, in studying models of geometric surfaces parametrized by a size\footnote{sometimes the size is controlled by multiple parameters.} $n$, a genus $g$, and their attached “counting quantities”, either volumes or cardinalities. Other important examples are combinatorial maps (or ribbon graphs), hyperbolic surfaces (in particular under the Weil--Petersson measure), and (half-)translation surfaces.

A recent concern is to study these models in the “large genus limit” (see e.g.\ ~\cite{mirzakhani_weil-petersson_2010,guth_pants_2011,angel_local_2013,mirzakhani_lengths_2019,BudzLouf,Delecroix_Goujard_Zograf_Zorich_2022,anantharaman_friedman-ramanujan_2025,hide_spectral_2025}). Many questions arise, such as typical geometric behavior of random objects, but an important goal is to obtain asymptotics for the counting numbers attached to these models, which is often a necessary condition to study geometric features.

Large/high genus asymptotics is an umbrella term that can mean either $g\to\infty$ and fixed $n$, or both $n,g\to\infty$. Here we focus on the second case,  a question of \emph{bivariate asymptotics}, which often entails some serious difficulties. Indeed, for the moment, when both $n,g\to\infty$ without restriction, true asymptotic equivalents are only reachable when explicit algorithms~\cite{Agg2} or simple recurrence formulas~\cite{EPLFW,BarLou} are available, and for other models, only partial asymptotics has been obtained~\cite{BudzLouf,ChapuyLoufWalsh2024PlancherelHurwitz}.

Fortunately, there is a regime where things become more tractable and formulas are in general simpler: when both $n,g\to\infty$, but $g$ grows much faster that $n$. This has already been established for  Weil--Petersson volumes~\cite{MirZog} and for intersection numbers and volumes of quadratic differentials~\cite{Agg}. 

In this paper, we continue this trend with \emph{double Hurwitz numbers} $h_g(\mu, \lambda)$, which count ramified covers of the sphere by higher genus surfaces, or equivalently factorizations in the symmetric group (see Section~\ref{s: def Hurwitz} for a precise definition). Note that our $h_g(\mu,\lambda)$ might differ from other papers by an automorphism factor. We show that if the genus becomes very large with respect to the size (in a quantifiable way), then we have a simple asymptotic equivalent for these numbers.

Our main result is the following. 

\begin{theorem}\label{thm_main}
Let $(g_n)$ be a sequence of integers such that $\frac{2g_n}{n}-\ln n\to\infty$ as $n\to \infty$. As $n\to \infty$, uniformly over all integer partitions $\mu,\lambda$ of $n$, setting $k_n = 2g_n-2+\ell(\l)+\ell(\mu)$, we have the following asymptotic equivalent: 
  \[h_{g_n}(\mu,\lambda)\sim \frac{2n!}{z_\mu z_\lambda}\binom{n}{2}^{k_n}, \]
where the length $\ell(\nu)$ and the normalization factor $z_\nu$ of a partition $\nu$ are defined in Section~\ref{s: integer partitions}.
\end{theorem}

Our condition on the growth of $g$ is optimal, see Remark~\ref{rem_tight}. 

One should note that when $\lambda,\mu,n$ are fixed and $g\to\infty$, the formula above was proven earlier~\cite{DYZ,DHR}. In this paper we show that it still holds true if $\nto$ slowly enough (compared to $g$), and we make explicit the fastest rate until which it still holds. These “fixed $n$ asymptotics” results have a very different proof method than ours, but they also apply to more general kinds of Hurwitz numbers~\cite{Yang-Hurwitz,ALR}. It would be interesting to see if our method can also be applied to extend these general results to the case where $\nto$ as well.

Our strategy of proof relies on interpreting double Hurwitz numbers as the number of paths of a given length, say $\ell$, between permutations with given cycle types in the Cayley graph of the symmetric group generated by transpositions,
that is, up to a scaling factor, as $\ell$-step transition probabilities for the simple random walk on this graph. This random walk is known as the random transposition shuffle famously studied by Diaconis and Shahshahani~\cite{DiaconisShahshahani1981}. Very precise results are known about the mixing time of this random walk (and generalizations), and this leads to estimates on \textit{unconnected} double Hurwitz numbers.
Then we need separate arguments to show that the probability of not being connected is very small, in order to convert our estimates on \textit{unconnected} double Hurwitz numbers into estimates on \textit{connected} double Hurwitz numbers.

\paragraph*{Acknowledgments}

BL thanks Guillaume Chapuy for joint works and enlightening discussions about Hurwitz numbers, and the connection with~\cite{DiaconisShahshahani1981}.

LT was supported by the ANR project LOUCCOUM (ANR-24-CE40-7809).
BL was partially supported by ANR CartesEtPlus (ANR-23-CE48-0018) and ANR HighGG
(ANR-24-CE40-2078-01).

\paragraph*{AI use statement} 

AI was not used, except to check for typos in the final version of the manuscript.

\section{Definitions}\label{sec_def}
\subsection{Integer partitions and contents of cells}\label{s: integer partitions}

Let $n\geq 2$. An integer partition of the integer $n$ is a tuple of integers $(\lambda_1, \ldots, \lambda_p)$ such that $\lambda_1\geq \ldots \lambda_p\geq 1$ and $\sum_{i=1}^p \lambda_i = n$. The number $p$ is called the length of $\lambda$ and is denoted by $\ell(\lambda)$. If $\lambda$ is a partition of $n$ with $m_i$ parts of size $i$, the size of the centralizer of a permutation with cycle type $\lambda$ is given by
\begin{equation}\label{eq:centralizer}
    z_{\lambda} = \prod_{i\geq 1} i^{m_i} m_i!.
\end{equation}
We denote the set of all (integer) partitions of $n$ by $\kY_n$, and may write $\lambda \vdash n$ for $\lambda \in \kY_n$.

An integer partition $\lambda = (\lambda_1, \ldots, \lambda_p) \vdash n$ can also be represented by a \textit{Young diagram}: this is a diagram of boxes with $\lambda_i$ boxes on the $i$-th row, as illustrated in Figure~\ref{fig: young diagram example}.

\begin{figure}[!ht]
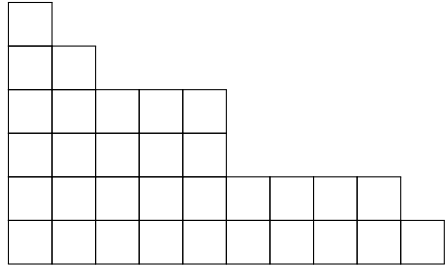

    \centering
{
\ytableausetup{boxsize=1.6em}
\begin{ytableau}
    \none & \\
    \none && \\
    \none &&&&& \\
    \none &&&&& \\
    \none &&&&&&&&& \\
    \none &&&&&&&&&& \\
\end{ytableau}
}
    \caption{The Young diagram of the integer partition $\lambda = (10,9,5,5, 2, 1)\vdash 32$: it has 10 cells on the first row, 9 on the second, 5 on the third, 5 on the fourth, 2 on the fifth, and 1 on the sixth.}
    \label{fig: young diagram example}
\end{figure}

Let $\lambda = (\lambda_1, \ldots, \lambda_p)$ be a partition. The number of boxes in the $j$-th column of $\lambda$ is denoted by $\lambda_j'$, and the conjugate partition is the partition $\lambda' = (\lambda'_1, \lambda'_2,\ldots, \lambda'_\ell)$. At the level of Young diagrams, $\lambda'$ is the reflection of $\lambda$ with respect to the diagonal.

If a box is in the $i$-th row and $j$-th column of a Young diagram, we say that it has coordinates $(i,j)$ and we might identify it with its coordinates.

The content of a box $(i,j)$ is the integer $j-i$, and the content $C_\lambda$ of a Young diagram $\lambda$ is the sum of the contents of its cells, that is,
\begin{equation}
    C_\lambda = \sum_{(i,j)\in \lambda} j-i,
\end{equation}
as illustrated in Figure~\ref{fig: content of a diagram}.

\begin{figure}[!ht]
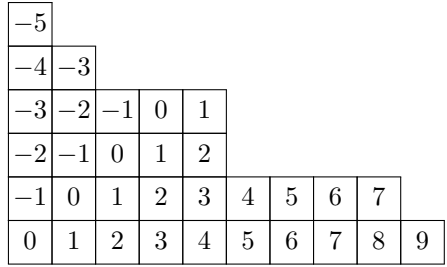

    \centering
{
\ytableausetup{boxsize=1.6em}
\begin{ytableau}
    \none &-5 \\
    \none &-4&-3 \\
    \none &-3&-2&-1&0&1 \\
    \none &-2&-1&0&1&2 \\
    \none &-1&0&1&2&3&4&5&6&7 \\
    \none &0&1&2&3&4&5&6&7&8&9 \\
\end{ytableau}
}
    \caption{The Young diagram of the integer partition $\lambda = (10,9,5,5,2,1)\vdash 32$ filled with its contents. The content of $\lambda$ is $C_\lambda = 45+28+3+1 - (15+6+1) = 55$.}
    \label{fig: content of a diagram}
\end{figure}

\subsection{Characters of symmetric groups and partitions}

Let $\kS_n$ be the symmetric group of index $n$. Its set of irreducible representations $\widehat{\kS_n}$ is in bijection with the set $\kY_n$ of integer partitions, and in what follows we identify representations with integer partitions: we write $\lambda\vdash n$ or $\lambda \in \kY_n$ or $\lambda \in \widehat{\kS_n}$ and consider that this is the same object.

For $\lambda\in \widehat{\kS_n}$ and $\sigma \in \kS_n$, we denote the dimension of the irreducible representation $\lambda$ by $d_\lambda$, the character of $\lambda$ evaluated at $\sigma$ by $\ch^\lambda(\sigma)$, and the associated renomalized character by $\chi^\lambda(\sigma) = \ch^\lambda(\sigma)/d_\lambda$.

In general, characters can be expressed using the complicated Murnaghan--Nakayama rule. However, for transpositions there are very simple formulas involving only the lengths of the rows and columns, or the contents: we have (denoting $\tau = (1 \; 2)$), for any $n\geq 2$ and $\lambda\vdash 2$,
\begin{equation}
    \chi^\lambda(\tau) = \frac{1}{\binom{n}{2}}\sum_{i\geq 1} \binom{\lambda_i}{2}-\binom{\lambda_i'}{2} = \frac{C_\lambda}{\binom{n}{2}}.
\end{equation}
There, the first equality is an old formula attributed to Frobenius, which can be found in \cite[Fact 3.D.2]{LivreDiaconis1988}, and the second, also classical, is obtained by rewriting the compensations in the sums differently.

\subsection{Hurwitz numbers}\label{s: def Hurwitz}

Given an integer $g\geq 0$ and two partitions $\l$ and $\mu$ of the same integer $n\geq 1$, and setting $k=2g-2+\ell(\mu)+\ell(\l)$, the (connected) double \emph{Hurwitz number} $h_g(\mu,\lambda)$ is the number of $(k+2)$-tuples of permutations $(\tau_1,\ldots,\tau_k,\sigma,\phi)$ of the symmetric group $\Sn$ such that each $\tau_i$ is a transposition, $\sigma$ has cycle type $\l$ and $\phi$ has cycle type $\mu$, and 
\begin{equation}\label{eq_def_hurwitz}
    \tau_1\cdot\tau_2\cdot\ldots\cdot \tau_k=\sigma\phi,
\end{equation}
with the additional condition that the action of the subgroup of $\Sn$ generated by the $(\tau_i)$ on $\ag 1,2,\ldots,n \ad$ is transitive. Equivalently, this means that the graph with vertex set $\ag 1,2,\ldots,n \ad$ and edge set given by the $(\tau_i)$ must be connected.
\begin{remark}
    Beware that definitions of Hurwitz numbers vary slightly depending on the articles. In particular, oftentimes they are defined as the number above, divided by $k!$. The terms in~\eqref{eq_def_hurwitz} may also be shuffled, but they give equivalent definitions by conjugation.
\end{remark}

Historically, Hurwitz numbers were introduced with a geometric definition~\cite{hurwitz_ueber_1891}: they count (equivalence classes of) branched covers of the sphere by a genus $g$ surface. In short, a branched cover (illustrated in Figure~\ref{fig:hurwitz}) is a holomorphic map $f$ between Riemann surfaces such that except for a finite number of special points, around any point $x$ of the genus $g$ surface “above” there is a local chart such that $f(x)=x$. The number $n$ in the definition above corresponds to the number of preimages of a typical point on the sphere. Non-typical points are called branch points and in the permutation model they are represented by the $\tau_i$, $\sigma$ and $\phi$ (more precisely, the cycle structure of the permutations encode the behaviour of their preimages).

\begin{figure}
    \centering
    \includegraphics[width=0.3\linewidth]{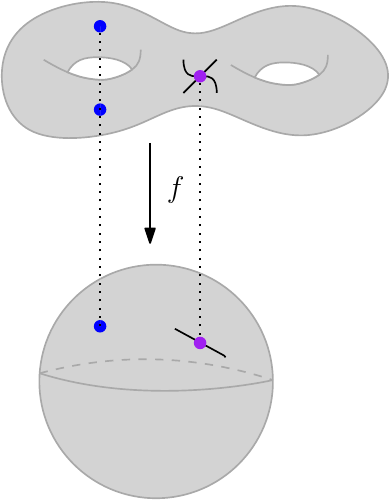}
    \caption{A branched cover, with $g=2$ and $n=2$. In blue, a typical point and its preimages. In purple, a branch point and its preimage.}
    \label{fig:hurwitz}
\end{figure}

\subsection{Random walks on symmetric groups}
Let $\cT_n$ be the conjugacy class of transpositions of $\kS_n$, and let $(X_t)_{t\geq 0}$ be the (discrete time) simple random walk on the (bipartite) Cayley graph $\Cay(\kS_n, \cT_n)$. In other words, this is the Markov chain on $\kS_n$ with transition probabilities $P(x,y) = \frac{1_{xy^{-1} \in \cT_n}}{\binom{n}{2}}$ for $x,y \in \kS_n$.
Given $x,y\in \kS_n$ and an integer $t\geq 0$, denote 
\begin{equation}
    p_t(x,y) = \bbP_x(X_t = y) = \bbP(X_t = y \mid X_0 = x).
\end{equation}
Note that $\binom{n}{2}^t p_t(\Id, \Id)$ is the number of paths of length $t$ from $\Id$ to $\Id$ in $\Cay(\kS_n, \cT_n)$, where $\Id$ is the identity element of $\kS_n$. 

Diaconis and Shahshahani \cite{DiaconisShahshahani1981} understood that this random walk\footnote{The original walk actually has some laziness, but this does not affect the analysis.}, also called the random transposition shuffle, could be studied using representation theory, and used this link to find precise mixing time estimates.

\section{Strategy of proof}
Let us first clarify how to interpret the right-hand-side of Theorem~\ref{thm_main}. There, 
\begin{itemize}
    \item $\frac{n!}{z_\l}$ is exactly the number of permutations of cycle type $\l$, and for any $\ell\geq 0$;
    \item $\binom{n}{2}^\ell$ counts the number of paths of length $\ell$ in $\Cay(\kS_n, \cT_n)$ starting from a given $\sigma\in\Sn$.
\end{itemize}
Therefore, by definition of Hurwitz numbers, we would like to show that for any $\sigma,\phi\in \Sn$ and $\ell$ large enough such that $\sgn(\sigma)\sgn(\phi)=(-1)^\ell$, the probability that the random walk started at $\phi$ ends at $\sigma$ after $\ell$ steps and that the associated graph is connected should be $\frac{2}{n!}$ (asymptotically). 

If we forget about the connectedness condition for a moment, this means that the endpoint of the walk should be asymptotically equidistributed among permutations of $\Sn$ with the right sign. Since the random walk is on a Cayley graph, we have $p_t(\phi,\sigma)=p_t(\Id,\phi^{-1}\sigma)$ for any $t$ and therefore we can focus on walks started at $\Id$. We will show the following result.

\begin{proposition}\label{prop_mixing}
    Let $(\omega_n)$ such that $\omega_n\to \infty$ as $n\to \infty$. As $n\to \infty$, uniformly over all $\sigma \in \kS_n$, and all $\ell \geq (n\ln n) + n\omega_n$, we have 
\begin{equation}
    \max_{\sigma\atop \sgn(\sigma)=(-1)^\ell}\left\lvert\frac{n!}{2}p_{\ell}(\Id,\sigma)-1\right\rvert\xrightarrow[n\to\infty]{} 0.
\end{equation}
\end{proposition}

Up to minor details, this amounts to bounding the $L^\infty$ mixing of the random transposition walk around $t=n\ln n$, which is equivalent to bounding the $L^2$ mixing time around time $t/2$ (for instance by \cite[Proposition 4.15]{LivreLevinPeres2019MarkovChainsAndMixingTimesSecondEdition}). Such bounds were first proved in the seminal paper of Diaconis and Shahshahani~\cite{DiaconisShahshahani1981}. Proposition~\ref{prop_mixing} will be proved in Section~\ref{s: proof of proposition prop mixing}.

The second ingredient of the proof is to show that the associated graph will be connected. It is now known~\cite{JainSawhney2024transpositionprofileotherproof} that w.h.p., as the mixing time is reached, the graph becomes connected (or, more precisely, there is no more isolated points, which in many random graph processes, is equivalent to being connected w.h.p.). However, there is a subtlety here: we want this to happen \textbf{conditionally on the endpoint of the walk}. In other words, we want to show the following proposition.

\begin{proposition}\label{prop_connected}
    Let $(\omega_n)$ such that $\omega_n\to \infty$ as $n\to \infty$. As $n\to \infty$, uniformly over all $\sigma \in \kS_n$, and all $\ell\geq (n\ln n) + n\omega_n$ we have
\[\bbP(X_\ell = \sigma \text{ and the graph is \textbf{not} connected }| X_0 = \Id)=o\bigpar{1/n!}.\]
\end{proposition}

Note that shortly after mixing (in $L^\infty$, i.e.\ shortly after time $n\ln n$), the probability of not being connected is a “large $o(1)$”, while the probability of hitting a given point is $2/n!$, therefore this proposition does not follow from previous existing results. In some sense, we want to show that the probability of not being connected is “uniformly spread among all possible endpoints”.

This result will be shown in Section~\ref{sec_connected}, using combinatorial arguments, and bounds of the number of walks with fixed endpoints and of length shorter than the mixing time, in the fashion of~\cite{ChapuyLoufWalsh2024PlancherelHurwitz}, but much coarser.

Together, Propositions~\ref{prop_mixing} and~\ref{prop_connected} directly imply Theorem~\ref{thm_main}.
\begin{remark}\label{rem_tight}
Propositions~\ref{prop_mixing} and~\ref{prop_connected} are tight in the sense that they are no longer true for smaller $\ell$, thus the condition $\frac{2g}{n}-\ln(n)\to\infty$ in Theorem~\ref{thm_main} is tight. 

But this is due to the fact that we consider \textbf{all} permutations of $\Sn$, restricting to certain classes would improve the result. An extremal example is given by long cycles: when restricted $\l =(n)$, Theorem~\ref{thm_main}
holds as soon as  $\frac{4g}{n} - \ln n \to \infty$, that is for genus twice lower than the general case.
\end{remark}

\section{Bounds on character sums}

\subsection{Hurwitz numbers and character sums}

Hurwitz numbers/transition probabilities for the transposition random walk can be written in terms of characters of the symmetric group. Indeed, there is a standard formula dating back to Frobenius: given $x,y\in \kS_n$ and an integer $t\geq 0$, we have
\begin{equation}\label{eq: Frobenius character formula transpositions at time t}
    p_t(x,y) = \frac{1}{n!}\sum_{\lambda\vdash n} \chi^\lambda(\tau)^t\ch^\lambda(x) \ch^\lambda(y) = \frac{1}{n!}\sum_{\lambda\vdash n} d_\lambda^2 \chi^\lambda(\tau)^t \chi^\lambda(x) \chi^\lambda(y).
\end{equation}

Let us now introduce a quantity that roughly approximates the formula above, while being slightly simpler and not containing negative terms, which will simplify the proofs.

\begin{equation}\label{eq: definition xi Xi Hurwitz numbers}
    \Hap_{n,\ell} = \frac{1}{n!}\sum_{\lambda\vdash n}d_\lambda^2|C_\lambda|^\ell \quad \quad \text{ and } \quad \quad \hap_{n,\ell} = \frac{1}{n!}\sum_{\lambda\vdash n}d_\lambda^2|\chi^\lambda(\tau)|^\ell.
\end{equation}
Note that we still have the relation $\Hap_{n,\ell} = \binom{n}{2}^\ell\hap_{n,\ell}$ for any $\ell$, and that for even $\ell$ we have $\hap_{n,\ell} = p_\ell(\Id, \Id)$.

\begin{lemma}\label{lem_bound_facto}
    Let $\sigma\in\Sn$ and $\ell \geq 0$. The number of factorizations of $\sigma$ in $\ell$ transpositions is bounded above by $\Hap_{n,\ell}$. Equivalently, we have $p_\ell(\Id, \sigma) \leq \hap_{n,\ell}$.
\end{lemma}
\begin{proof}
This follows from \eqref{eq: Frobenius character formula transpositions at time t}, the triangle inequality, and the fact that $|\chi^\lambda(\phi)| \leq 1$ for any $\phi \in \kS_n$.
\end{proof}
\begin{lemma}\label{lem: monotonie hap n ell}
    Let $n\geq 1$. Then $\ell \mapsto \hap_{n,\ell}$ is weakly decreasing.
\end{lemma}
\begin{proof}
    For every $\lambda\vdash n$ we have $|\chi^\lambda(\tau)| \leq 1$, so the result holds by definition of $\hap_{n,\ell}$.
\end{proof}

\subsection{Bounds on character sums from mixing times}

Precise bounds on character sums -- including those that define $\hap_{n,\ell}$ and $\Hap_{n,\ell}$ in \eqref{eq: definition xi Xi Hurwitz numbers} -- are well-known in the context of mixing times. In particular, the second named author recently proved an easy-to-cite version of these results which we now recall. This result holds for conjugacy classes with order $n$ fixed points, which includes transpositions, but for simplicity of notation we recall only its specialization to transpositions. We also do not need the specific shape of the limiting function.
\begin{proposition}[{\cite[Theorem 4.1, specialized to transpositions]{Teyssier2025ProfilesConjugacyInvariant}}]\label{prop: cv profiles in ell two rewritten}
    There exists a continuous decreasing function $f\du \bbR \to (0, \infty)$ tending to 0 at $+\infty$, such that the following holds.
Let $a\in \bbR$. As $n\to \infty$, recalling that we write $\tau$ for any transposition of $\kS_n$, we have
\begin{equation}
   \sum_{\lambda \in \widehat{\kS_n}\backslash\ag [n], [1^n]\ad} d_\lambda^2 \bg \chi^\lambda(\tau)\bd^{ \lf n(\ln n + a)\rf}\to f(a).
\end{equation}
\end{proposition}

\begin{corollary}\label{cor: cv profiles in ell two rewritten}
    Let $(\omega_n)$ such that $\omega_n\to \infty$ as $n\to \infty$. As $n\to \infty$, uniformly over all $\ell \geq (n\ln n) + n\omega_n$, we have
\begin{equation}
   \sum_{\lambda \in \widehat{\kS_n}\backslash\ag [n], [1^n]\ad} d_\lambda^2 \bg \chi^\lambda(\tau)\bd^{\ell}\to 0. 
\end{equation}
\end{corollary}
\begin{proof}
   This follows from Proposition~\ref{prop: cv profiles in ell two rewritten} and the monotonicity of $t\mapsto \sum_{\lambda} d_\lambda^2 \bg \chi^\lambda(\tau)\bd^{t}$.
\end{proof}
We note that a similar estimate could also be obtained by slightly modifying the proofs of \cite{DiaconisShahshahani1981}, or also as a consequence of Propositions 3 and 4 from \cite{Hough2016}. 

We can now prove a bound that will be useful when studying the connectedness probability in Section~\ref{sec_connected}.
\begin{proposition}\label{prop_bound_Hnl}
As $n\to \infty$ we have:
\begin{enumerate}
    \item uniformly for all $\ell \geq n\ln n$,
\begin{equation}
\hap_{n,\ell} \leq \frac{O(1)}{n!};
\end{equation}
\item
uniformly for all $\ell\leq n\ln n$, 
\begin{equation}
    \hap_{n,\ell} \leq e^{-\ell} e^{O(n)}.
\end{equation}
\end{enumerate}
\end{proposition}
\begin{proof}
\begin{enumerate}
     \item
By monotonicity (Lemma~\ref{lem: monotonie hap n ell}), it is enough to prove the result at $\ell = \lc n \ln n \rc$, which follows immediately from the definition of $\hap_{n,\ell}$ (in \eqref{eq: definition xi Xi Hurwitz numbers}) and Proposition~\ref{prop: cv profiles in ell two rewritten}.
    \item 
We start by recalling two basic estimates: for any $\lln$, setting $r := n- \lambda_1$, we have
\begin{align}
&d_\l\leq \binom{n}{r}d_{\l_{\geq 2}}\leq \binom{n}{r}\sqrt{r!}\label{ineq_dl},\\
&\chi^\l(\tau)\leq 1-\frac{r}{n}. \label{ineq_content_a}
\end{align}
These two inequalities are quite standard, they can be found for instance in~\cite{LivreDiaconis1988}. More precisely, there \eqref{ineq_content_a} is stated as \cite[Fact 3.D.2]{LivreDiaconis1988}, the first inequality of \eqref{ineq_dl} is stated as \cite[Fact 3.D.1]{LivreDiaconis1988}, and the second inequality of \eqref{ineq_dl}, i.e.\ $d_{\l_{\geq 2}} \leq \sqrt{r!}$, is immediate since $d_{\l_{\geq 2}}^2 \leq \sum_{\mu \vdash r} d_\mu^2 = r!$ (the equality is the dimension equality in the Fourier isomorphism, recalled as \cite[Corollary 2.D.2 (a)]{LivreDiaconis1988})

Now take $\lambda\vdash n$ such that $C_\l>0$ and $\ell \leq n\ln n$. Write $r = n-\lambda_1$ and $\alpha = \frac{\ell}{n\ln n}$. By~\eqref{ineq_dl}, we have
\begin{equation}
    \frac{d_\lambda^2}{n!} \leq \binom{n}{r}^2 \frac{r!}{n!} =  e^{O(n)}n^{-(n-r)}.
\end{equation}
On the other hand, by~\eqref{ineq_content_a}, we have $\chi^\lambda(\tau) \leq 1-\frac{r}{n} \leq e^{-r/n}$, so 
\begin{equation}
   \chi^\lambda(\tau)^\ell\leq e^{-r\ell/n}= n^{-\alpha r}.
\end{equation}
It follows, using that $e^{\ell} = n^{\alpha n}$, that
\begin{equation}
     \frac{d_\lambda^2}{n!} \chi^\lambda(\tau)^\ell \leq e^{O(n)}\pg e^{-\ell} n^{\alpha n}\pd  n^{-(n-r)} n^{-\alpha r} = e^{O(n)}e^{-\ell} n^{-(n-r)(1-\alpha)} \leq e^{O(n)}e^{-\ell}.
\end{equation}
This concludes the proof since there are $e^{O(n)}$ partitions of the integer $n$.
\qedhere
\end{enumerate}
\end{proof}

\subsection{Proof of Proposition~\ref{prop_mixing}}\label{s: proof of proposition prop mixing}

As we now see, we can already prove that if we do not impose any connectivity assumption, the probability that, if $\omega_n \to \infty$, then the end-point of a random path of length $n\ln n + n\omega_n$ in $\Cay(\kS_n, \cT_n)$ started at $\Id$ is a given permutation does essentially not depend on the chosen permutation (given its parity class).

\begin{proof}[Proof of Proposition~\ref{prop_mixing}.]
    Let $n\geq 2$, $t\geq 0$, and $\sigma\in \kS_n$ with sign $(-1)^t$. \\
    Since $\chi^{[1^n]}(\tau)^t\chi^{[1^n]}(\sigma) = (-1)^t(-1)^t = 1$, by \eqref{eq: Frobenius character formula transpositions at time t} we have
    \begin{equation}
        \bg n!p_t(\Id,\sigma) - 2\bd = \bg \sum_{\lambda \in \widehat{\kS_n}\backslash\ag [n], [1^n]\ad} d_\lambda^2 \chi^\lambda(\tau)^t  \chi^\lambda(\Id)\chi^\lambda(\sigma)\bd.
    \end{equation}
By the triangle inequality, and using that $\bg \chi^\lambda(\sigma)\bd \leq 1$ and $\chi^\lambda(\Id)=1$, we deduce that
\begin{equation}
        \bg n!p_t(\Id,\sigma) - 2\bd \leq \sum_{\lambda \in \widehat{\kS_n}\backslash\ag [n], [1^n]\ad} d_\lambda^2 \bg \chi^\lambda(\tau)\bd ^t.
    \end{equation}
The desired result then follows from Corollary~\ref{cor: cv profiles in ell two rewritten}.
\end{proof}

\section{Probability of not being connected}\label{sec_connected}

\subsection{A combinatorial upper bound for unconnected Hurwitz numbers}
Let $n\geq 4$ and $\ell \geq 2$.
We say that a factorization $\tau_1 \cdots \tau_\ell = \sigma$ of a permutation $\sigma\in \kS_n$ into $\ell$ transpositions is \textit{splittable} if there exist subsets $\emptyset \subsetneq S \subsetneq [n]$  and $\emptyset \subseteq I \subseteq [\ell]$
such that for every $i\in I$, $\tau_i$ transposes two elements of $S$ and for every $i\in [\ell] \backslash I$, $\tau_i$ transposes two elements of $[n]\backslash S$. (To avoid double counting when swapping the roles of $S$ and $S^c$, we add the restriction that $S$ is of size at most $n/2$, and if $|S| = n/2$, we impose that it contains the element 1.) We denote the number of splittable factorizations into $\ell\geq 2$ transpositions of a permutation $\sigma\in \kS_n$ by $H^{\textup{spl}}_{n,\ell}(\sigma)$, and set
\begin{equation}\label{eq: def h split}
\begin{split}
     h^{\textup{spl}}_{n,\ell}(\sigma) & = \binom{n}{2}^{-\ell} H^{\textup{spl}}_{n,\ell}(\sigma) \\
     & = \bbP(X_\ell = \sigma \text{ and the graph is \textbf{not} connected after $\ell$ steps}| X_0 = \Id).
\end{split}
\end{equation}
We finally set
\begin{equation}\label{eq: def q n ell}
    q_{n,\ell} = n! \max_{\sigma \in \kS_n} h^{\textup{spl}}_{n,\ell}(\sigma).
\end{equation}
The goal of this section is to show the following proposition, which directly implies Proposition~\ref{prop_connected}.

\begin{proposition}\label{prop: final o(1) asymptotic bound on qnl} Let $(\omega_n)$ such that $\omega_n\to \infty$ as $n\to \infty$. As $n\to \infty$, uniformly over all $\ell \geq (n\ln n) + n\omega_n$, we have

\begin{equation}
    q_{n,\ell}=o(1).
\end{equation}
\end{proposition}

Given $n\geq 4$ and $ \ell \geq 2$, we define $\cS = \cS(n,\ell)$ to be the set of 4-tuples $(n_1, n_2, \ell_1, \ell_2)$ of non-negative integers such that $n_1+n_2 = n, \ell_1+ \ell_2 = n$, $n_1\leq n/2$, and $0<\ell_1<\ell$. 

We start with a combinatorial bound on $q_{n,\ell}$

\begin{proposition}
    Let $n \geq 4$ and $\ell \geq 2$ be integers.
    We have
\begin{equation}\label{eq: bound qnl with hap prop}
    \qnl\leq n!\sum_{(n_1, n_2, \ell_1, \ell_2)\in\cS}\binom{n}{n_1}\binom{\ell}{\ell_1} \frac{\binom{n_1}{2}^{\ell_1}\binom{n_2}{2}^{\ell_2}}{\binom{n}{2}^\ell}\hap_{n_1, \ell_1}\hap_{n_2, \ell_2}.
\end{equation}

\end{proposition}
\begin{proof}
    Let $\sigma\in \kS_n$ be a permutation that admits a splittable factorization of length $\ell$. We want to bound how many such factorizations exist, that is, to bound $H^{\textup{spl}}_{n,\ell}(\sigma)$. Any such factorization can be obtained by applying the following steps.
\begin{itemize}
    \item First pick $(n_1, n_2, \ell_1, \ell_2) \in \cS$.
    \item Then pick $S$ and $I$ as in the definition, with respective sizes $n_1$ and $\ell_1$, (and having $S$ contain 1 if $n_1= n/2$). There are at most $\binom{n}{n_1}\binom{\ell}{\ell_1}$ such choices.
    \item Define $\alpha_1\in \kS_n$ and $\alpha_2\in \kS_n$ to be the restrictions of $\sigma$ to the sets $S$ and $[n]\backslash S$ respectively. 
    \item Pick transpositions $\tau_i$ transposing elements of $S$ for $i\in I$ and transposing elements of $[n]\backslash S$ for $i\in [\ell]\backslash I$, such that $\prod_{i\in I} \tau_i = \alpha_1$ and $\prod_{i\in [\ell]\backslash I} \tau_i = \alpha_2$, where the products are ordered. Up to relabeling, this counts the number of transposition paths of length $\ell_1$ (resp.\ $\ell_2$) from the identity to $\alpha_1$ (resp.\ $\alpha_2$) in $\kS_{n_1}$ (resp.\ $\kS_{n_2}$), which is upper bounded by $\Hap_{n_1, \ell_1}$ (resp.\ $\Hap_{n_2, \ell_2}$) by Lemma~\ref{lem_bound_facto}. 
\end{itemize}
This shows that 
\begin{equation}
    H^{\textup{spl}}_{n,\ell}(\sigma) \leq \sum_{(n_1, n_2, \ell_1, \ell_2)\in \cS}\binom{n}{n_1}\binom{\ell}{\ell_1}\Hap_{n_1,\ell_1}\Hap_{n_2,\ell_2}.
\end{equation}
We deduce, taking the maximum over $\sigma$, that
\begin{equation}\label{eq: bound qnl with Hap inter}
   \qnl = n! \binom{n}{2}^{-\ell} \max_{\sigma\in \kS_n}H^{\textup{spl}}_{n,\ell}(\sigma)\leq  n! \binom{n}{2}^{-\ell}\sum_{(n_1, n_2, \ell_1, \ell_2)\in \cS}\binom{n}{n_1}\binom{\ell}{\ell_1}\Hap_{n_1,\ell_1}\Hap_{n_2,\ell_2}.
\end{equation}
Rewriting \eqref{eq: bound qnl with Hap inter}, using that $\Hap_{n_i, \ell_i} = \hap_{n_i, \ell_i} \binom{n_i}{2}^{\ell_i}$ for $i\in \ag 1, 2\ad$, we obtain the desired result.
\end{proof}

\subsection{Proof of Proposition~\ref{prop: final o(1) asymptotic bound on qnl}}
We now want to bound the right hand side of \eqref{eq: bound qnl with hap prop}. We first prove the following technical lemma.

\begin{lemma}\label{lem: study of an auxiliary function}
   Let $\alpha \in (0,2)$, and $\beta>1$. There exists a constant $C_{\alpha, \beta}$ such that $(\beta/k)^kn^{\alpha k^2/n}\leq C_{\alpha, \beta}$ for all integers $n\geq 2$ and $1\leq k\leq n/2$.
\end{lemma}
\begin{proof}
Let $n$ be an integer, which we may assume is large enough, and in particular such that $n>e\beta$. We want to show that the function $f$ defined by $f(x) = (\beta/x)^xn^{\alpha x^2/n}$ for $x\in [1, n/2]$ is bounded.
Clearly, $\max_{x\in [1,e\beta]}f(x) = O_{\alpha,\beta}(1)$ as $n \to \infty$.
Now, for $x\in [e\beta,n/2]$, set $g(x) = \frac{\ln(\beta/x)}{x} + \alpha \frac{\ln n}{n}$, so that $\ln f(x) = x^2 g(x)$. 
For $x\in [e\beta,n/2]$, we have $g'(x) = ((\ln x)-\ln(e\beta)/x^2 \geq 0$, so $\max g = g(n/2) = -(2-\alpha)\frac{\ln n}{n} +O_{\alpha, \beta}\pg \frac{1}{n} \pd$ which is negative for $n$ large enough. Hence for $n$ large enough and all $x\in [e\beta, n/2]$ we have $\ln f(x) \leq 0$, i.e.\ $f(x)\leq 1$. The result follows. 
\end{proof}
 We now bound the right hand side of \eqref{eq: bound qnl with hap prop}. We first decompose the sum into simpler sub-sums. Below the symbol $\sqcup$ is used for disjoint unions.
\begin{lemma}\label{lem: decomposition cS into 3}
    Let $n\geq 4$ and $\ell \geq n\ln n$. Then $\cS = \cS_1 \sqcup \cS_2 \sqcup \cS_3$, where
\begin{align*}
    & \cS_1 = \ag (n_1,n_2,\ell_1,\ell_2) \in \cS(n,\ell) \mid n_1\ln n_1\leq \ell_1 \textup{ and  } n_2\ln n_2\leq \ell_2 \ad, \\
    &  \cS_2 = \ag (n_1,n_2,\ell_1,\ell_2) \in \cS(n,\ell) \mid n_1\ln n_1> \ell_1 \textup{ and  } n_2\ln n_2\leq \ell_2 \ad, \\
     &  \cS_3 = \ag (n_1,n_2,\ell_1,\ell_2) \in \cS(n,\ell) \mid n_1\ln n_1\leq \ell_1 \textup{ and  } n_2\ln n_2> \ell_2 \ad.
\end{align*}
\end{lemma}
\begin{proof}
Let $\cS_4 = \ag (n_1,n_2,\ell_1,\ell_2) \in \cS(n,\ell) \mid n_1\ln n_1> \ell_1 \textup{ and  } n_2\ln n_2> \ell_2 \ad$. Clearly, $\cS = \bigsqcup_{i=1}^4 \cS_i$. However, since $\ell \geq n\ln n$, by convexity of $x\mapsto x\ln x$ it is not possible to have both $n_1 \ln n_1 > \ell_1$ and $n_2 \ln n_2 > \ell_2$, that is, we have $\cS_4 = \emptyset$, which concludes the proof.
\end{proof}

We are now ready to prove Proposition~\ref{prop: final o(1) asymptotic bound on qnl}.

\begin{proof}[Proof of Proposition~\ref{prop: final o(1) asymptotic bound on qnl}]
    By Lemma~\ref{lem: decomposition cS into 3} and changing variables, we may rewrite
\begin{equation}\label{eq: bound qnl with hap lem sum S i}
         n!\sum_{(n_1, n_2, \ell_1, \ell_2)\in\cS}\binom{n}{n_1}\binom{\ell}{\ell_1} \frac{\binom{n_1}{2}^{\ell_1}\binom{n_2}{2}^{\ell_2}}{\binom{n}{2}^\ell}\hap_{n_1, \ell_1}\hap_{n_2, \ell_2} = \sum_{i=1}^3 \sum_{k=1}^{\lfloor n/2 \rfloor} S_i(k),
\end{equation}   
where for $i\in \ag 1, 2, 3 \ad$ and $1\leq k\leq \lfloor n/2 \rfloor$, 
\begin{equation}
    S_i(k) =  n!\sum_{s \du (k, n-k, s, \ell - s)\in\cS_i}\binom{n}{k}\binom{\ell}{s} \frac{\binom{k}{2}^{s}\binom{n-k}{2}^{\ell - s}}{\binom{n}{2}^\ell}\hap_{k, s}\hap_{n-k, \ell-s}.
\end{equation}
Our goal is therefore to bound $\sum_{k=1}^{\lfloor n/2 \rfloor} S_i(k)$ for each $i\in \ag 1, 2, 3 \ad$. In this proof we use the notation $f_n \lesssim g_n$, which has the same meaning as $f_n = O(g_n)$.

\paragraph*{First case: $i=1$.}
Applying Proposition~\ref{prop_bound_Hnl} (a) to both $\hap_{k, s}$ and $\hap_{n-k, \ell-s}$ in the sum, we obtain that uniformly for all $1\leq k\leq \lfloor n/2 \rfloor$,
\begin{equation}\label{eq: some intermediate bound on S 1}
\begin{split}
    S_1(k) & \lesssim n!\sum_{s \du (k, n-k, s, \ell - s)\in \cS_1}\binom{n}{k}\binom{\ell}{s} \frac{\binom{k}{2}^{s}\binom{n-k}{2}^{\ell - s}}{\binom{n}{2}^\ell}\frac{1}{k!}\frac{1}{(n-k)!} \\
       & = \binom{n}{k}^2\sum_{s \du (k, n-k, s, \ell - s)\in\cS_1}\binom{\ell}{s}\frac{\binom{k}{2}^{s}\binom{n-k}{2}^{\ell-s}}{\binom{n}{2}^{\ell}} \\
       & \lesssim \binom{n}{k}^2 \sum_{s \du (k, n-k, s, \ell - s)\in\cS_1}\binom{\ell}{s}\bigpar{(k/n)^2}^{s}\bigpar{(1-k/n)^2}^{\ell-s},
\end{split}
        \end{equation}
and therefore, using also that $\binom{n}{k} \leq (en/k)^k$, and that “$1-x\leq e^{-x}$”, that
\begin{equation}
    S_1(k) \lesssim \binom{n}{k}^2 \bigpar{(k/n)^2+(1-k/n)^2}^{\ell} = \binom{n}{k}^2 \bigpar{1-2(k/n)(1-k/n)}^{\ell} 
    \leq (en/k)^{2k} a_k^\ell,
\end{equation}
where $a_k = e^{-2(k/n)(1-k/n)} \leq e^{-k/n}$.
Now, uniformly for $1\leq k\leq n/2$, we have 
\begin{equation}
    a_k^{n\ln n} = \pg e^{-2(k/n)(1-k/n)}\pd^{n\ln n} = e^{-2k \ln n} e^{(2k^2/n)\ln n} = n^{-2k} n^{2k^2/n},
\end{equation}
so, using Lemma~\ref{lem: study of an auxiliary function}, 
\begin{equation}
    (en/k)^{2k} a_k^{n\ln n} = (e/k)^{2k} n^{2k^2/n} = \pg (e/k)^{k} n^{k^2/n} \pd^2 = O(1),
\end{equation}
and therefore,
\begin{equation}
    (en/k)^{2k} a_k^\ell = (en/k)^{2k} a_k^{n\ln n} a_k^{n\omega_n} \lesssim a_k^{n\omega_n} \leq \pg e^{-k/n}\pd^{n\omega_n} = \pg e^{-\omega_n} \pd^k.
\end{equation}
It follows that $\sum_{k=1}^{\lfloor n/2 \rfloor} S_1(k)\lesssim \sum_{k=1}^{\lfloor n/2 \rfloor} \pg e^{-\omega_n} \pd^k = o(1)$.

\paragraph*{Second case: $i=2$.}
For this case we bound $\hap_{k, s}$ and $\hap_{n-k, \ell-s}$ using Proposition~\ref{prop_bound_Hnl} (b) and (a) respectively (instead of (a) for both in the first case). This gives (in the sum) $\hap_{k, s} \leq C^k e^{-s} = \frac{1}{k!} \pg k!C^k e^{-s}\pd$ for some universal constant $C>0$, and $\hap_{n-k, \ell-s} = O(1)/(n-k)!$.
Following the same steps as in \eqref{eq: some intermediate bound on S 1}, we obtain that uniformly for all $1\leq k \leq n/2$, 

\begin{equation}
\begin{split}
    S_2(k) & \lesssim \binom{n}{k}^2 \sum_{s \du (k, n-k, s, \ell - s)\in\cS_2}\binom{\ell}{s}\bigpar{\frac{k}{n}}^{2s}\bigpar{1-\frac{k}{n}}^{2(\ell-s)} \pg k!C^k e^{-s}\pd \\
    & = C^k \binom{n}{k} \frac{n!}{(n-k)!}\sum_{s \du (k, n-k, s, \ell - s)\in\cS_2}\binom{\ell}{s}\bigpar{\frac{(k/n)^2}{e}}^{s}\bigpar{(1-k/n)^2}^{\ell-s},
\end{split}
\end{equation}
and therefore that
\begin{equation}
    S_2(k) \lesssim C^k\binom{n}{k} n^k  \bigpar{\frac{(k/n)^2}{e}+(1-k/n)^2}^\ell \leq \pg \frac{eCn^2}{k} \pd^k b_k^\ell,
\end{equation}
where $b_k = e^{-2k/n + (1+1/e)(k/n)^2} \leq e^{-k/(2n)}$.

We conclude proceeding as in the first case that $\sum_{k=1}^{\lfloor n/2 \rfloor} S_2(k)\lesssim \sum_{k=1}^{\lfloor n/2 \rfloor} \pg e^{-\omega_n /2} \pd^k = o(1)$.

\paragraph*{Third case: $i=3$.}
For this last case, we use Proposition~\ref{prop_bound_Hnl} (a) and (b) respectively we obtain that $\hap_{k, s} = O(1)/k!$ and $\hap_{n-k, \ell-s} = e^{O(n)}e^{-(\ell-s)}$, which leads to, uniformly for $1\leq k\leq \lfloor n/2 \rfloor$,
\begin{equation}
     S_3(k) \leq e^{O(n)} \binom{n}{k}^2(n-k)!  \bigpar{(k/n)^2+\frac{(1-k/n)^2}{e}}^\ell.
\end{equation}
Moreover, uniformly for $1\leq k\leq \lfloor n/2 \rfloor$ we have $(k/n)^2+\frac{(1-k/n)^2}{e} 
\leq 1/e$ and $\binom{n}{k}^2(n-k)! \leq 4^n n^n$, so that
\begin{equation}
    S_3(k) \leq e^{O(n)} e^{n\ln n} e^{-\ell} = e^{O(n)} e^{-n\omega_n}.
\end{equation}
We deduce that $\sum_{k=1}^{\lfloor n/2 \rfloor} S_3(k) \leq n e^{O(n)}e^{-n\omega_n} = o(1)$, which concludes the proof of the third case, and hence of the proof.
\end{proof}

\subsection{Proofs of Proposition~\ref{prop_connected} and Theorem~\ref{thm_main}}

We claimed earlier that Proposition Proposition~\ref{prop: final o(1) asymptotic bound on qnl} is stronger than Proposition~\ref{prop_connected}, and that Theorem~\ref{thm_main} follows from Propositions~\ref{prop_mixing} and~\ref{prop_connected}. We now prove these claims.

\begin{proof}[Proof of Proposition~\ref{prop_connected}]
    By definition (see \eqref{eq: def h split} and \eqref{eq: def q n ell}), for any $n, \ell$ and $\sigma\in \kS_n$, we have
\begin{equation}
    \bbP(X_\ell = \sigma \text{ and the graph is \textbf{not} connected }| X_0 = \Id) = h^{\textup{spl}}_{n,\ell}(\sigma) \leq \max_{\phi \in \kS_n}h^{\textup{spl}}_{n,\ell}(\phi) =  \frac{q_{n,\ell}}{n!}.
\end{equation}
The result therefore follows immediately from Proposition~\ref{prop: final o(1) asymptotic bound on qnl}.
\end{proof}

\begin{proof}[Proof of Theorem~\ref{thm_main}]
Set $\omega_n = \frac{2g_n}{n}-\ln n$ for $n\geq 1$. By assumption we have $\omega_n \to \infty$ as $n\to \infty$. Recall that there are $n!/z_\lambda$ permutations of $\kS_n$ with a given cycle type $\lambda$.

As $n\to \infty$, uniformly over all $\sigma, \phi \in \kS_n$, and all $\ell \geq 2g_n$, since the number of $\ell$-tuples of transpositions satisfying \eqref{eq_def_hurwitz} and the connectivity assumption below \eqref{eq_def_hurwitz} is by definition equal to $\binom{n}{2}^\ell \bbP(X_\ell = \sigma \text{ and the graph \textbf{is connected} }| X_0 = \Id)$, by Propositions~\ref{prop_mixing} and~\ref{prop_connected} it is asymptotically equivalent to $\frac{2}{n!} \binom{n}{2}^\ell$.
Summing this estimate over all $\sigma, \phi$ with respective given cycle types $\lambda, \mu$, we deduce that
\begin{equation}
    h_g(\lambda, \mu) \sim \frac{n!}{z_\lambda}\frac{n!}{z_\mu} \frac{2}{n!} \binom{n}{2}^{k_n} = \frac{2n!}{z_\lambda z_\mu} \binom{n}{2}^{k_n},
\end{equation}
where $k_n = 2g_n - 2 + \ell(\lambda) + \ell(\mu) \geq 2g_n$, which concludes the proof.
\end{proof}

\footnotesize{
\bibliographystyle{alpha}
\bibliography{biblioBL}

@article {ChapuyLoufWalsh2024PlancherelHurwitz,
    AUTHOR = {Chapuy, Guillaume and Louf, Baptiste and Walsh, Harriet},
     TITLE = {Random partitions under the {P}lancherel-{H}urwitz measure,
              high-genus {H}urwitz numbers and maps},
   JOURNAL = {Ann. Probab.},
  FJOURNAL = {The Annals of Probability},
    VOLUME = {52},
      YEAR = {2024},
    NUMBER = {4},
     PAGES = {1225--1252},
      ISSN = {0091-1798,2168-894X},
   MRCLASS = {60C05 (05A05 05A17 60B15 60D05)},
  MRNUMBER = {4766834},
MRREVIEWER = {Yuval\ Filmus},
       DOI = {10.1214/23-aop1651},
       URL = {https://doi.org/10.1214/23-aop1651},
}

@misc{BarLou,
 author = {Simon Barazer and Baptiste Louf},
 title = {High genus one part monotone {Hurwitz} numbers},
 year = {2026},
 howpublished = {Preprint, {arXiv}:2604.01190 [math.{CO}] (2026)},
 url = {https://arxiv.org/abs/2604.01190},
 arXiv = {arXiv:2604.01190}
}

@article{hurwitz_ueber_1891,
	title = {Ueber {Riemann}'sche {Flächen} mit gegebenen {Verzweigungspunkten}},
	volume = {39},
	issn = {1432-1807},
	url = {https://doi.org/10.1007/BF01199469},
	doi = {10.1007/BF01199469},
	language = {de},
	number = {1},
	urldate = {2026-09-17},
	journal = {Mathematische Annalen},
	author = {Hurwitz, A.},
	month = mar,
	year = {1891},
	pages = {1--60},
}

@article{EMS,
 author = {Eynard, Bertrand and Mulase, Motohico and Safnuk, Bradley},
 title = {The {Laplace} transform of the cut-and-join equation and the {Bouchard}-{Mari{\~n}o} conjecture on {Hurwitz} numbers},
 fjournal = {Publications of the Research Institute for Mathematical Sciences, Kyoto University},
 journal = {Publ. Res. Inst. Math. Sci.},
 issn = {0034-5318},
 volume = {47},
 number = {2},
 pages = {629--670},
 year = {2011},
 language = {English},
 doi = {10.2977/PRIMS/47},
 zbMATH = {5931054},
 Zbl = {1225.14022}
}

@article{Agg2,
 author = {Aggarwal, Amol},
 title = {Large genus asymptotics for volumes of strata of abelian differentials},
 fjournal = {Journal of the American Mathematical Society},
 journal = {J. Am. Math. Soc.},
 issn = {0894-0347},
 volume = {33},
 number = {4},
 pages = {941--989},
 year = {2020},
 language = {English},
 doi = {10.1090/jams/947},
 zbMATH = {7268735},
 Zbl = {1452.14025}
}

@article{Agg,
 author = {Aggarwal, Amol},
 title = {Large genus asymptotics for intersection numbers and principal strata volumes of quadratic differentials},
 fjournal = {Inventiones Mathematicae},
 journal = {Invent. Math.},
 issn = {0020-9910},
 volume = {226},
 number = {3},
 pages = {897--1010},
 year = {2021},
 language = {English},
 doi = {10.1007/s00222-021-01059-9},
 zbMATH = {7433733},
 Zbl = {1480.14020}
}

@article{MirZog,
 author = {Mirzakhani, Maryam and Zograf, Peter},
 title = {Towards large genus asymptotics of intersection numbers on moduli spaces of curves},
 fjournal = {Geometric and Functional Analysis. GAFA},
 journal = {Geom. Funct. Anal.},
 issn = {1016-443X},
 volume = {25},
 number = {4},
 pages = {1258--1289},
 year = {2015},
 language = {English},
 doi = {10.1007/s00039-015-0336-5},
 url = {hdl.handle.net/21.11116/0000-0004-1461-8},
 zbMATH = {6483699},
 Zbl = {1327.14136}
}

@article{angel_local_2013,
	title = {The local limit of unicellular maps in high genus},
	volume = {18},
	number = {86},
	journal = {Electron. Commun. Probab.},
	author = {Angel, O. and Chapuy, Guillaume and Curien, Nicolas and Ray, Gourab},
	year = {2013},
	pages = {1--8},
}

@article{mirzakhani_weil-petersson_2010,
	title = {On {Weil}-{Petersson} {Volumes} and {Geometry} of {Random} {Hyperbolic} {Surfaces}},
	journal = {Proceedings of the ICM 2010},
	author = {Mirzakhani, Maryam},
	year = {2010},
}

@misc{EPLFW,
 author = {Andrew Elvey-Price and Wenjie Fang and Baptiste Louf and Michael Wallner},
 title = {Bivariate asymptotics via random walks: application to large genus maps},
 year = {2026},
 howpublished = {Preprint, {arXiv}:2506.06924 [math.{CO}] (2026)},
 url = {https://arxiv.org/abs/2506.06924},
 arXiv = {arXiv:2506.06924}
}

@article{Delecroix_Goujard_Zograf_Zorich_2022, title={Large genus asymptotic geometry of random square-tiled surfaces and of random multicurves}, volume={230}, ISSN={0020-9910, 1432-1297}, DOI={10.1007/s00222-022-01123-y},  note={arXiv:2007.04740 [math.GT]}, number={1}, journal={Inventiones mathematicae}, author={Delecroix, Vincent and Goujard, Elise and Zograf, Peter and Zorich, Anton}, year={2022}, month=oct, pages={123–224} }

@article{BudzLouf,
 author = {Budzinski, Thomas and Louf, Baptiste},
 title = {Local limits of uniform triangulations in high genus},
 fjournal = {Inventiones Mathematicae},
 journal = {Invent. Math.},
 issn = {0020-9910},
 volume = {223},
 number = {1},
 pages = {1--47},
 year = {2021},
 language = {English},
 doi = {10.1007/s00222-020-00986-3},
 zbMATH = {7306969},
 Zbl = {1461.05182}
}

@article{mirzakhani_lengths_2019,
	title = {Lengths of closed geodesics on random surfaces of large genus},
	volume = {94},
	issn = {0010-2571},
	url = {https://doi.org/10.4171/cmh/477},
	doi = {10.4171/cmh/477},
	number = {4},
	journal = {Commentarii Mathematici Helvetici. A Journal of the Swiss Mathematical Society},
	author = {Mirzakhani, Maryam and Petri, Bram},
	year = {2019},
	mrnumber = {4046008},
	pages = {869--889},
}

@article{guth_pants_2011,
	title = {Pants decompositions of random surfaces},
	volume = {21},
	issn = {1016-443X},
	url = {https://doi.org/10.1007/s00039-011-0131-x},
	doi = {10.1007/s00039-011-0131-x},
	number = {5},
	journal = {Geometric and Functional Analysis},
	author = {Guth, Larry and Parlier, Hugo and Young, Robert},
	year = {2011},
	mrnumber = {2846383},
	pages = {1069--1090},
}

@misc{hide_spectral_2025,
	title = {Spectral gap with polynomial rate for {Weil}-{Petersson} random surfaces},
	url = {http://arxiv.org/abs/2508.14874},
	doi = {10.48550/arXiv.2508.14874},
	
	urldate = {2026-07-10},
	publisher = {arXiv},
	author = {Hide, Will and Macera, Davide and Thomas, Joe},
	month = nov,
	year = {2025},
	note = {arXiv:2508.14874 [math.SP]},
}

@misc{anantharaman_friedman-ramanujan_2025,
	title = {Friedman-{Ramanujan} functions in random hyperbolic geometry and application to spectral gaps {II}},
	url = {http://arxiv.org/abs/2502.12268},
	doi = {10.48550/arXiv.2502.12268},
	urldate = {2026-09-02},
	publisher = {arXiv},
	author = {Anantharaman, Nalini and Monk, Laura},
	month = jun,
	year = {2025},
	note = {arXiv:2502.12268 [math.MG]},
}

@article{Ok,
 author = {Okounkov, Andrei},
 title = {Toda equations for {Hurwitz} numbers},
 fjournal = {Mathematical Research Letters},
 journal = {Math. Res. Lett.},
 issn = {1073-2780},
 volume = {7},
 number = {4},
 pages = {447--453},
 year = {2000},
 language = {English},
 doi = {10.4310/MRL.2000.v7.n4.a10},
 zbMATH = {1539310},
 Zbl = {0969.37033}
}

@misc{Yang-Hurwitz,
 author = {Chenglang Yang},
 title = {The structures of simple {Hurwitz} numbers and monotone {Hurwitz} numbers with varying genus},
 year = {2025},
 howpublished = {Preprint, {arXiv}:2503.01920 [math.{CO}] (2025)},
 url = {https://arxiv.org/abs/2503.01920},
 arXiv = {arXiv:2503.01920}
}

@misc{ALR,
 author = {Davide Accadia and Danilo Lewa{\'n}ski and Giulio Ruzza},
 title = {On the large genus of {Hurwitz} numbers},
 year = {2026},
 howpublished = {Preprint, {arXiv}:2604.12880 [math.{AG}] (2026)},
 url = {https://arxiv.org/abs/2604.12880},
 arXiv = {arXiv:2604.12880}
}

@misc{DHR,
 author = {Norman Do and Jian He and Heath Robertson},
 title = {The structure of {Hurwitz} numbers with fixed ramification profile and varying genus},
 year = {2024},
 howpublished = {Preprint, {arXiv}:2409.06655 [math.{CO}] (2024)},
 url = {https://arxiv.org/abs/2409.06655},
 arXiv = {arXiv:2409.06655}
}

@article{DYZ,
 author = {Dubrovin, Boris and Yang, Di and Zagier, Don},
 title = {Classical {Hurwitz} numbers and related combinatorics},
 fjournal = {Moscow Mathematical Journal},
 journal = {Mosc. Math. J.},
 issn = {1609-3321},
 volume = {17},
 number = {4},
 pages = {601--633},
 year = {2017},
 language = {English},
 url = {www.mathjournals.org/mmj/2017-017-004/2017-017-004-003.html},
 zbMATH = {7069274},
 Zbl = {1411.14059}
}

@article{ELSV,
 author = {Ekedahl, Torsten and Lando, Sergei and Shapiro, Michael and Vainshtein, Alek},
 title = {Hurwitz numbers and intersections on moduli spaces of curves.},
 fjournal = {Inventiones Mathematicae},
 journal = {Invent. Math.},
 issn = {0020-9910},
 volume = {146},
 number = {2},
 pages = {297--327},
 year = {2001},
 language = {English},
 doi = {10.1007/s002220100164},
 zbMATH = {1777217},
 Zbl = {1073.14041}
}

@article {DiaconisShahshahani1981,
    AUTHOR = {Diaconis, Persi and Shahshahani, Mehrdad},
     TITLE = {Generating a random permutation with random transpositions},
   JOURNAL = {Z. Wahrsch. Verw. Gebiete},
  FJOURNAL = {Zeitschrift f\"{u}r Wahrscheinlichkeitstheorie und Verwandte
              Gebiete},
    VOLUME = {57},
      YEAR = {1981},
    NUMBER = {2},
     PAGES = {159--179},
      ISSN = {0044-3719},
   MRCLASS = {60C05 (60J15)},
  MRNUMBER = {626813},
MRREVIEWER = {Lars Holst},
       DOI = {10.1007/BF00535487},
       URL = {https://doi.org/10.1007/BF00535487},
}

@article {JainSawhney2024transpositionprofileotherproof,
    AUTHOR = {Jain, Vishesh and Sawhney, Mehtaab},
     TITLE = {Hitting time mixing for the random transposition walk},
    YEAR = {2024+},
   JOURNAL = {Arxiv preprint \textup{\url{https://arxiv.org/abs/2410.23944}}},
       URL = {},
}

@article {Hough2016,
    AUTHOR = {Hough, Bob},
     TITLE = {The random {$k$} cycle walk on the symmetric group},
   JOURNAL = {Probab. Theory Related Fields},
  FJOURNAL = {Probability Theory and Related Fields},
    VOLUME = {165},
      YEAR = {2016},
    NUMBER = {1-2},
     PAGES = {447--482},
      ISSN = {0178-8051},
   MRCLASS = {60J10 (20C30 60B15)},
  MRNUMBER = {3500276},
MRREVIEWER = {Sara Brofferio},
       DOI = {10.1007/s00440-015-0636-6},
       URL = {https://doi.org/10.1007/s00440-015-0636-6},
}

@book {LivreLevinPeres2019MarkovChainsAndMixingTimesSecondEdition,
    AUTHOR = {Levin, David A. and Peres, Yuval},
     TITLE = {Markov chains and mixing times},
      NOTE = {Second edition of [ MR2466937],
              With contributions by Elizabeth L. Wilmer,
              With a chapter on ``Coupling from the past'' by James G. Propp
              and David B. Wilson},
 PUBLISHER = {American Mathematical Society, Providence, RI},
      YEAR = {2017},
     PAGES = {xvi+447},
      ISBN = {978-1-4704-2962-1},
   MRCLASS = {60J10 (60-01 60B15 60C05 60J27 60K35 68U20 82C22)},
  MRNUMBER = {3726904},
       DOI = {10.1090/mbk/107},
       URL = {https://doi.org/10.1090/mbk/107},
}

@book {LivreDiaconis1988,
    AUTHOR = {Diaconis, Persi},
     TITLE = {Group representations in probability and statistics},
    SERIES = {Institute of Mathematical Statistics Lecture Notes---Monograph
              Series},
    VOLUME = {11},
 PUBLISHER = {Institute of Mathematical Statistics, Hayward, CA},
      YEAR = {1988},
     PAGES = {vi+198},
      ISBN = {0-940600-14-5},
   MRCLASS = {60-02 (20C99 62-02)},
  MRNUMBER = {964069},
MRREVIEWER = {Philippe Bougerol},
}

@article {Teyssier2025ProfilesConjugacyInvariant,
    AUTHOR = {Teyssier, Lucas},
     TITLE = {Cutoff profiles for conjugacy invariant random walks on symmetric groups},
    JOURNAL = {arXiv preprint \textup{\url{https://arxiv.org/abs/2605.28770}}},
        YEAR = {2026+},
       URL = {},
}
}
\end{document}